\documentclass[letterpaper,11pt]{amsart}

\usepackage[all]{xy}                        %

\CompileMatrices                            

\UseTips                                    

\input xypic
\usepackage[bookmarks=true]{hyperref}       

\usepackage{amssymb,latexsym,amsmath,amscd}
\usepackage{xspace}

\usepackage{graphicx}

\theoremstyle{plain}
\newtheorem{theorem}{Theorem}[section]
\newtheorem*{theorem*}{Theorem}
\newtheorem{proposition}[theorem]{Proposition}
\newtheorem{corollary}[theorem]{Corollary}
\newtheorem{lemma}[theorem]{Lemma}

\theoremstyle{definition}

\newtheorem{remark}[theorem]{Remark}

\newcommand{\enm}[1]{\ensuremath{#1}}          %
\newcommand{\op}[1]{\operatorname{#1}}
\newcommand{\cal}[1]{\mathcal{#1}}

\newcommand{\CC}{\enm{\mathbb{C}}}

\newcommand{\ZZ}{\enm{\mathbb{Z}}}

\newcommand{\PP}{\enm{\mathbb{P}}}

\newcommand{\Bb}{\enm{\cal{B}}}

\newcommand{\Ee}{\enm{\cal{E}}}
\newcommand{\Ff}{\enm{\cal{F}}}
\newcommand{\Gg}{\enm{\cal{G}}}

\newcommand{\Ii}{\enm{\cal{I}}}

\newcommand{\Nn}{\enm{\cal{N}}}
\newcommand{\Oo}{\enm{\cal{O}}}
\newcommand{\Pp}{\enm{\cal{P}}}

\newcommand{\Rr}{\enm{\cal{R}}}

\renewcommand{\phi}{\varphi}
\renewcommand{\theta}{\vartheta}
\renewcommand{\epsilon}{\varepsilon}

\newcommand{\Pic}{\op{Pic}}

\renewcommand{\to}[1][]{\xrightarrow{\ #1\ }}

\newcommand{\old}[1]{}

\begin{document}

\title[Globally generated vector bundles on a smooth quadric threefold]{Globally generated vector bundles of rank $2$\\ on a smooth quadric threefold}
\author{E. Ballico, S. Huh and F. Malaspina}
\address{Universit\`a di Trento, 38123 Povo (TN), Italy}
\email{edoardo.ballico@unitn.it}
\address{Sungkyunkwan University, Suwon 440-746, Korea}
\email{sukmoonh@skku.edu}
\address{Politecnico di Torino, Corso Duca degli Abruzzi 24, 10129 Torino, Italy}
\email{francesco.malaspina@polito.it}
\thanks{The second author is supported by Basic Science Research Program 2012-0002904 through NRF funded by MEST. The third author is supported by the framework of PRIN 2010/11 "Geometria delle varieta' algebriche", cofinanced by MIUR}

\begin{abstract}
We investigate the existence of globally generated vector bundles of rank 2 with $c_1\leq 3$ on a smooth quadric threefold and determine their Chern classes. As an automatic consequence, every rank 2 globally generated vector bundle on $Q$ with $c_1=3$ is an odd instanton up to twist.
\end{abstract}

\maketitle
\section{Introduction}
Globally generated vector bundles on projective varieties play an important role in algebraic geometry. If they are non-trivial they must have strictly positive first Chern class.  Globally generated vector bundles on projective spaces with low first Chern class have been investigated in several papers. If $c_1(\Ee)=1$ then it is easy to see that modulo trivial summands we have only $\Oo_{\PP^n}(1)$ and $T\PP^n(-1)$.  The classification of rank $r$ globally generated vector bundles with $c_1=2$ is settled in \cite{SU}. In \cite{huh} the second author carried out the case of rank two  with $c_1=3$ on $\mathbb P^3$ and in \cite{ce} the authors continued the study until $c_1\leq 5$. This classification was extended to any rank in \cite{m} and to any $\mathbb P^n$ ($n\geq 3$) in \cite{am} and \cite{SU2}.  In \cite{e} are shown the possible Chern classes of rank two globally generated vector bundles
on  $\mathbb P^2$.

Let $Q$ be a smooth quadric threefold over an algebraically closed field of characteristic zero. The aim of this paper is to investigate the existence of globally generated vector bundles of rank $2$ on $Q$ with $c_1\leq 3$.
We use an old method of associating to a rank 2 vector bundle on $Q$ a curve in $Q$, and relate properties of the bundle to properties of the curve.
  If $\Ee$ is globally generated, he admits an exact sequence,
$$0\to \Oo_Q \to \Ee \to \Ii_C(c_1) \to 0,$$
where $C$ is a smooth  curve of degree $c_2(\Ee)$ and genus $g$:
We prove the following theorem:

\begin{theorem}\label{mt}
 There exists an indecomposable and globally generated vector bundle $\Ee$ of rank $2$ on $Q$ with the Chern classes $(c_1, c_2)$,  $c_1\leq 3$ in the following cases:

\begin{enumerate}
\item $(c_1=1, c_2=1)$, $\Ee$ is the spinor bundle $\Sigma$ and $C$ is a line.
\item  $(c_1=2, c_2=4)$, $\Ee$ is a pull-back of a null-correlation bundle on $\mathbb {P}^3$ twisted by $1$ and $C$ is the disjoint union of two conics.
    \item $(c_1=3, c_2=5)$ and $\Ee\cong\Sigma(1)$.
    \item $(c_1=3, 6\leq c_2\leq 9)$, $\Ee$ is the homology of a monad
    $$0\rightarrow \Oo_Q^{\oplus c_2-5}(1) \rightarrow \Sigma^{\oplus c_2-4}(1) \rightarrow \Oo_Q(2)^{\oplus c_2-5} \rightarrow 0,$$ and $C$ is a smooth elliptic curve of degree $c_2$.
\end{enumerate}

\end{theorem}

A typical way to construct a vector bundle on $Q$ is by restriction of a vector bundle on $\PP
^4$ or by a pull-back of a vector bundle on $\PP^3$ along a linear projection from $Q$ to $\PP^3$. The spinor bundle $\Sigma$ is not obtained by either one of these ways and it plays an important role in describing the globally generated vector bundle on $Q$.
In fact, from the classification we observe that every rank two globally generated vector bundle on $Q$ with $c_1=3$ is, up to twist, an odd instanton that is the cohomology of a monad involving $\Sigma$ (see \cite{faenzi}).

In Sect.2 we set up basic computations and deal with the case $c_1=1$ as a preliminary case. In Sect.3 we prove that every indecomposable and globally generated vector bundle of rank 2 on $Q$ with $c_1=2$ is a pull-back of a null-correlation bundle on $\PP^3$ twisted by $1$. In Sect.4, 5 and 6, we deal with the case of $c_1=3$. First we determine the possible second Chern class $c_2$ using the Liaison theory, that is $5\leq c_2 \leq 9$. In each case we prove the existence of globally generated vector bundle of rank $2$. In Sect.4 we explain the case of $c_2=5,6,7$ based on the result \cite{OS} about the moduli spaces of rank $2$ vector bundles on $Q$. The critical part of this paper is on the existence of globally generated vector bundles of rank $2$ with $c_2=8$ and $9$. It is equivalent to the existence of a smooth elliptic curve $C$ of degree $c_2$ whose ideal sheaf twisted by $3$ is globally generated. The main ingredient is the result in \cite{hh} with replacing $\PP^3$ by $Q$, by which we can deform a nodal reducible curve of degree $c_2$ constructed in a suitable way to a smooth elliptic curve that we need.
In order to have a complete classification we need an answer to the following question:\\
{\bf Question:} Are the moduli spaces $\mathfrak{M}(3,8)$ and $\mathfrak{M}(3,9)$  irreducible?\\
In the last section we classify rank two globally generated bundles on higher dimensional quadrics.\\

The second author would like to thank Politecnico di Torino, especially the third author for warm hospitality.

\section{Preliminaries}
Let $Q$ be a smooth quadric hypersurface in $\PP^4$. Then we have
$$\Pic (Q)=H^2(Q, \ZZ)=\ZZ h,$$
where $h$ is the class of a hyperplane section. Moreover, the cohomology ring $H^*(Q, \ZZ)$ is generated by $h$, a line $l\in H^4(Q, \ZZ)$ and a point $p\in H^6(Q, \ZZ)$ with the relations: $h^2=2l~,~ h\cdot l=p~,~ h^3=2p$.

Let $\Ee$ be a coherent sheaf of rank $r$ on $Q$. Then we have \cite{OS}:
\begin{align*}
c_1(\Ee(k))&=c_1+kr\\
c_2(\Ee(k))&=c_2+2k(r-1)c_1+2k^2{r\choose 2}\\
c_3(\Ee(k))&=c_3+k(r-2)c_2+2k^2{r-1\choose 2}+2k^3{r\choose 3} \\
\chi(\Ee)~~&=(2c_1^3-3c_1c_2+3c_3)/6+3(c_1^2-c_2)/2+13c_1/6+r,
\end{align*}
where $(c_1, c_2, c_3)$ is the Chern classes of $\Ee$. In particular, when $\Ee$ is a vector bundle of rank 2 with $c_1=-1$, we have
$$\chi(\Ee)=1-c_2~~,~~ \chi(\Ee(1))=6-2c_2~~,~~ \chi(\Ee(-1))=0,$$
$$\chi(\mathcal{E}nd(\Ee))=7-6c_2.$$

Let $\mathfrak{M}(c_1, c_2)$ be the moduli space of stable vector bundles of rank 2 on $Q$ with the Chern classes $(c_1, c_2)$.

\begin{proposition}\cite{sierra}\label{prop1}
Let $\Ee$ be a globally generated vector bundle of rank $r$ on $Q$ such that $H^0(\Ee(-c_1))\not= 0$, where $c_1$ is the first Chern class of $\Ee$. Then we have
$$\Ee\simeq \Oo_Q^{\oplus (r-1)}\oplus \Oo_Q(c_1).$$
\end{proposition}

 Now assume that $\Ee$ is a globally generated vector bundle of rank 2 on $Q$. Then $\Ee$ admits an exact sequence
\begin{equation}\label{eqa1}
0\rightarrow \Oo_Q \rightarrow \Ee \rightarrow \Ii_C (c_1) \rightarrow 0,
\end{equation}

where $C$ is a smooth curve on $Q$. Notice that $\omega_C \simeq \Oo_C(-3+c_1)$ and $c_2(\Ee)=\deg (C)$.

If $l$ is a line on $Q$, then $\Ee|_l$ is also globally generated, which means in particular that $c_1(\Ee)\geq 0$.

 \begin{remark}\label{u0}
$\Ee \cong \Oo _Q(c_1) \oplus \Oo _Q$ if and only if $C =\emptyset$.
\end{remark}

We briefly recall Ottaviani's construction of the spinor
bundle (\cite{O1}). Let $G(2,4)$ denote the Grassmannian of all $2$-dimensional linear
subspaces of
$\mathbb {K}^4$. By
using the geometry of the variety of all $1$-dimensional linear subspaces of $Q$ it is possible to construct a morphism
$s: Q \to G(2,4)$.  Set $\Sigma (-1):= s^\ast U$ where $U$ is the universal quotient bundle of $G(2,4)$.
 $\Sigma$ is called spinor bundle on $Q$, he  is a globally generated vector bundle of rank $2$ with the Chern classes $(c_1, c_2)=(1,1)$. Moreover he admits the following canonical exact sequence
$$0\to \Sigma^\vee \to \Oo_Q^{\oplus 4} \to \Sigma \to 0.$$
For a general section $s\in H^0(\Sigma)$, we have an exact sequence
$$0\to \Oo_Q \stackrel{s}{\to} \Sigma \to \Ii_l (1) \to 0$$
where $l$ is a line on $Q$. It gives us an identification between $\PP H^0(\Sigma)\cong \PP^3$ and a family of the lines on $Q$.

\begin{proposition}\label{u1}
We have $c_1(\Ee) =1$ if and only if either $\Ee \cong \Oo_Q(1)\oplus \Oo _Q$ or $\Ee$ is the spinor bundle.
\end{proposition}

\begin{proof} Both $\Oo_Q(1)\oplus \Oo _Q$ and the spinor bundle are globally generated and have the prescribed Chern classes. Hence it is sufficient to consider the case when $C \ne \emptyset$ by the remark \ref{u0}. Since $C$ is smooth and $\mathcal {I}_C(1)$ is globally generated, it is contained in a codimension
$2$ linear section of $Q$. Hence $C$ is either a smooth conic or a line. In both cases
$C$ is ACM. From the sequence (\ref{eqa1}) we get that $\Ee$ is an ACM vector bundle. Hence $\Ee$ is either decomposable or a twist of the spinor bundle by Theorem 3.5 in \cite{O}. Since $\Ee$ is globally generated and $c_1(\Ee)=1$, we get
that either $\Ee \cong \Oo_Q(1)\oplus \Oo _Q$ or $\Ee$ is the spinor bundle.
\end{proof}

\section{Rank two globally generated bundles with $c_1=2$}
In this section we prove the following result.
\begin{proposition}\label{u2}
Let $\Ee$ be a rank $2$ vector bundle with $c_1(\Ee) =2$. $\Ee$ is globally generated
if and only if $\Ee$ is either isomorphic to
\begin{enumerate}
\item $\Oo_Q(a)\oplus  \Oo_Q(2-a)$ with $a=0,1$ or
\item  a pull-back of a null-correlation bundle on $\mathbb {P}^3$ twisted by $1$
\end{enumerate}
\end{proposition}

\begin{lemma}\label{u3}
Let $C = E_1\sqcup E_2\sqcup J \subset Q$ be a curve with $E_1$ and $E_2$ smooth conics and $J\ne \emptyset$. Then $\Ii_C(2)$ is not globally generated.
\end{lemma}

\begin{proof}
Assume that $\Ii_C(2)$ is globally generated. Let $H \subset \PP^4$ denote a general hyperplane containing
$E_1$. Set $Q':= H\cap Q$. $Q'$ is a smooth quadric surface, the scheme $Z:= H\cap (E_2\cup J)$ is
a zero-dimensional scheme of degree at least 3 and $Z\cap E_1=\emptyset$. Since $\Ii_C(2)$
is globally generated, so is $\Ii_{E_1\cup Z,Q'}(2,2)$; since $E_1$ is a curve
of type $(1,1)$ on $Q'$ and $Z\cap E_1=\emptyset$, so $\Ii_Z(1,1)$ is globally generated.
But this is absurd since $\deg (Z)\ge 3$.
\end{proof}

\qquad {\emph {Proof of Proposition \ref{u2}.}} Let us assume that the vector bundle $\Ee$ on $Q$ is globally generated with $c_1(\Ee)=2$ and then it fits into the exact sequence (\ref{eqa1}) with $c_1=2$ and $C$ a smooth curve.
Since $\omega_C\simeq \Oo_C(-1)$, $C$ is a disjoint union of conics, i.e. $C=C_1\cup \cdots \cup C_r$, where each $C_i$ is a conic. By Remark \ref{u0} we have
$\Ee \cong \Oo _Q(2)\oplus \Oo _Q$ if and only if $r=0$. Now assume $r>0$. Lemma \ref{u3}
gives $r\in \{1,2\}$. As the first case let us assume that $r=1$. Since a smooth conic is ACM, the sequence (\ref{eqa1}) gives that
$\Ee$ is ACM. Hence we have $\Ee \cong  \Oo_Q(1)\oplus  \Oo_Q(1)$ by Theorem 3.5 in \cite{O}.

Now let us assume that $r=2$. Let $M_i\subset \mathbb {P}^4$ be the plane spanned by $C_i$. Since $M_i\cap Q = C_i$ and $C_1\cap C_2 =\emptyset$, the set $M_1\cap M_2$ cannot be a line. Hence $M_1\cap M_2$ is a point, say $P$. Since $M_i\cap Q = C_i$ and $C_1\cap C_2 =\emptyset$, we have $P\notin Q$.
The linear projection $\ell _P: \mathbb {P}^4\setminus \{P\} \to \mathbb {P}^3$
send each $E_i$ onto a line $L_i$. In Prop. 3.2 of \cite{SW}, it was shown that all bundles as an extension in (\ref{eqa1}) are pull-back from bundles from extensions
\begin{equation}\label{a2}
0 \to \mathcal {O}_{\mathbb {P}^3} \to \Ff \to \mathcal {I}_{L_1\sqcup L_2}(2)\to 0
\end{equation}
on $\mathbb {P}^3$, in which $\Ff$ is a null-correlation bundle twisted by $1$.

\section{Rank two globally generated bundles with $c_1=3$}

Let $\Ee$ be a globally generated vector bundle of rank 2 on $Q$ with the Chern classes $(c_1, c_2)$, $c_1\geq 3$. It fits into the exact sequence (\ref{eqa1}).

\begin{lemma}
The Chern classes $(c_1, c_2)$ of $\Ee$ satisfies the following inequality for $c_1\geq 3$:
$$c_2\leq \frac{2}{3}(2c_1+1)(c_1-1).$$
In particular, if $c_1=3$, we have $c_2\leq 9$.
\end{lemma}
\begin{proof}
Since $\Ii_C(c_1)$ is globally generated, so there are two hypersurfaces of degree $c_1$ in $Q$ whose intersection is a curve $X$ containing $C$.
Let $Y$ be a curve such that $X=C+Y$. If $\Ee$ does not split, then $Y$ is not empty so we have the exact sequence of liaison:
$$0\rightarrow \Ii_X (c_1) \rightarrow \Ii_C(c_1) \rightarrow \omega_Y(3-c_1) \rightarrow 0.$$
Since $\Ii_C(c_1)$ is globally generated, so is $\omega_Y(3-c_1)$. It implies that
$$\deg(\omega_Y(3-c_1))=2g'-2+d'(3-c_1)\geq0,$$
$(d'=\deg (Y), g'=g(Y))$ and so $g' \geq 0$. On the other hand, by liaison, we have
$$g'-g=\frac 12 (d'-d)(2c_1-3)$$
and $d'=2c_1^2-d$, $2g-2=d(c_1-3)$ since $\omega_C(3-c_1)=\Oo_C$. Here $d=\deg (C)$ and $g=g(C)$.
Thus we have $g'=1+\frac{d(c_1-3)}{2}+(c_1^2-d)(2c_1-3)\geq0$ and so
$$c_2 \leq \frac{2(2c_1^3-3c_1^2+1)}{3c_1-3}=\frac{2}{3} (2c_1+1)(c_1-1).$$
\end{proof}

Assume that $\Ee$ is a globally generated vector bundle of rank 2 on $Q$ with $c_1=3$ that fits into the sequence:
$$0\rightarrow \Oo_Q \rightarrow \Ee \rightarrow \Ii_C(3) \rightarrow 0,$$
where $C$ is a smooth curve with $\deg (C)=c_2(\Ee)$.

By Proposition \ref{prop1}, if $H^0(\Ee(-3))\not=0$, then $\Ee$ is isomorphic to $\Oo_Q\oplus \Oo_Q(3)$, which is globally generated. So let us assume that $H^0(\Ee(-3))=0$.

As the first case, let us assume that $H^0(\Ee(-2))\not= 0$, i.e. $\Ee$ is unstable.

\begin{proposition}
Let $\Ee$ be a globally generated unstable vector bundle of rank 2 on $Q$ with $c_1(\Ee)=3$. Then $\Ee$ is a direct sum of line bundles. In other words, we have
$$\Ee \simeq \Oo_Q(a)\oplus \Oo_Q(3-a)~~~~,~~~~a=0,1.$$
\end{proposition}
\begin{proof}
Note that $h^0(\Ii_C(1))>0$ and so $C$ is contained in the complete intersection of $Q$ and two hypersurfaces of degree $1$ and $3$ inside $\PP^4$ that is the ambient space containing $Q$. In particular, we have $\deg (C)\leq 6$. On the other hand, from a section in $H^0(\Ee(-2))$, we have a sequence
$$0\rightarrow \Oo_Q (2) \rightarrow \Ee \rightarrow \Ii_{C'}(1) \rightarrow 0.$$
If $C'$ is empty, then $\Ee$ is isomorphic to $\Oo_Q(1)\oplus \Oo_Q(2)$. If not, the degree of $C'$, which is $c_2(\Ee)-4$ is at least $1$. Thus $\deg (C)=c_2(\Ee)$ is either 5 or 6.
Note that $\omega_C\simeq \Oo_C$ and so $C$ consists of smooth elliptic curves. If $\deg (C)=5$, then $C$ is a quintic elliptic curve contained in a hyperplane section of $Q$ which is a quadric surface $Q_2$. In the case when $Q_2$ is smooth, let $(a,b)$ is the bidegree of $C$ as a divisor of the quadric surface and then we have $\deg (C)=5=a+b$ and $g(C)=1=ab-a-b+1$. But it is impossible. Similarly we can show that $\deg(C)= 6$ is not possible. If $Q_2$ is a quadric surface cone, we can show that it is impossible by Exercise V.2.9 of \cite{Hartshorne}.
\end{proof}

Assume now that $H^0(\Ee(-2))=0$, i.e. $\Ee$ is stable. By the Bogomolov inequality, we have $c_2(\Ee)\geq 5$. Recall that $c_2(\Ee)\leq 9$.

\begin{proposition}\cite{OS}
Every vector bundle in $\mathfrak{M}(3, c_2)$ with $c_2=5,6$ is globally generated. In the case of $c_2=7$, the general vector bundle is globally generated.
\end{proposition}
\begin{proof}
Note that $\Ee$ fits into the sequence
\begin{equation}\label{seq2}
0\rightarrow \Oo_Q(1) \rightarrow \Ee \rightarrow \Ii_Z(2) \rightarrow 0.
\end{equation}
with $Z$ a locally complete intersection, $\deg (Z) =c_2-4$ and $\omega_Z\simeq \Oo_Z(-2)$.

If $c_2(\Ee)=5$, then $Z$ is a line. From the sequence (\ref{seq2}) we get that $\Ee$ is ACM. Since $\Ee$ is stable, we get $\Ee \simeq \Sigma(1)$ \cite{O}. Obviously $\Sigma (1)$ is globally generated.

If $c_2(\Ee)=6$, then $\Ee$ is the cohomology of the following monad (Remark 4.8 in \cite{OS}):
$$0\rightarrow \Oo_Q(1) \rightarrow \Sigma(1)^{\oplus 2} \rightarrow \Oo_Q(2) \rightarrow 0.$$
In this case, $Z$ is either 2 disjoint lines on $Q$ or a line with multiplicity 2. It is also known to be globally generated.

If $c_2(\Ee)=7$, then a general vector bundle $\Ee$ in $\mathfrak{M}(3,7)$ can be shown to be globally generated using the `Castelnuovo-Mumford criterion' (Theorem 5.2 in \cite{OS}). The same result also follows from the Lemma \ref{b1} with $C=A$. Since $h^0(\Ii_C(2)) =0$
and $h^0(C,\Oo_C(2)) = 14 = h^0(\Oo _Q(2))$, (\ref{eqa1}) gives $h^1(\Ee(-1)) =0$. Since $C$ is an elliptic curve,
we have $h^1(C,\Oo _C(1)) =0$. Hence (\ref{eqa1}) gives $h^2(\Ee(-2)) = h^1(C,\Oo _C(1)) =0$.
We have $h^3(\Ee(-3)) = h^0(\Ee^\vee ) = h^0(\Ee(-3)) =0$. Hence the `Castelnuovo-Mumford criterion' gives that $\Ee$ is globally generated.
\end{proof}

\section{Case $(c_1, c_2)=(3,8)$}

Let $\Ff=\Ee(-2)$. We can compute:\
$$
\left\{
       \begin{array}{ll}
                                           H^1(\Ff(t))=0 \text{  for  } t\leq -1\\
                                             \chi (\Ff \otimes \Ff^\vee)=-17 \\
                                             \chi (\Ff)=-\chi (\Ff(-2))=-3\\
                                             \chi (\Ff(1))=-\chi (\Ff(-3))=-2\\
                                             \chi (\Ff(-1))=0\\
                                             \chi (\Ff(2))=7
                                           \end{array}
                             \right. $$
Then the cohomology table for $\Ff=\Ee(-2)$ is as follows:\\
\begin{equation}\label{t1}
\begin{tabular}{|c|c|c|c|c|c|}
  \hline
  0 & 0 & 0 & 0 & 0 & 0 \\
  2 & 3 & 0 & 0 & 0 & 0 \\
  0 & 0 & 0 & 3 & 2 & $b$ \\
  0 & 0 & 0 & 0 & 0 & $a$ \\
   \hline
   \hline
  -3& -2& -1 & 0 & 1 & 2 \\
  \hline
\end{tabular}\end{equation}
with $a-b=7$.

\begin{proposition}\label{monad}
The vector bundles $\Ee$ in $\mathfrak{M}(3,8)$ with $H^0(\Ee (-1))=0$ is the cohomology of the following monad:
$$ 0\to \Oo_Q(1)^{\oplus 3}\to  \Sigma(1)^{\oplus 4}\to \Oo_Q(2)^{\oplus 3} \to 0$$
\end{proposition}
\begin{proof}Let us consider the sequence killing $H^1(\Ff)$:
\begin{equation} 0\to \Ff\to \Bb\to \Oo_Q^{\oplus 3}\to 0
\end{equation}
$H^2_*(\Ff)\cong H^2_*(\Bb)$ and  $H^1(\Ff)=H^0(\Ff)=0$ since the map $H^0(\Oo_Q^{\oplus 3})\to H^1(\Ff) $ is an isomorphism. Moreover $h^3(\Bb(-3))=h^3(\Oo_Q(-3)^{\oplus 3})=3$. So the cohomology table for $\Bb$ and $\Bb^{\vee}$ are as follows respectively :\\
\begin{center}
\begin{tabular}{|c|c|c|c|}
  \hline
  3 & 0 & 0 & 0   \\
  2 & 3 & 0 & 0   \\
  0 & 0 & 0 & 0   \\
  0 & 0 & 0 & 0   \\
   \hline
   \hline
  -3& -2& -1 & 0   \\
  \hline
\end{tabular}\quad\quad
\begin{tabular}{|c|c|c|c|}
  \hline
  0 & 0 & 0 & 0  \\
  0 & 0 & 0 & 0  \\
  0 & 0 & 3 & 2  \\
  0 & 0 & 0 & 3  \\
   \hline
   \hline
  -3& -2& -1 & 0  \\
  \hline

\end{tabular}
\end{center}

Now let us consider the sequence killing $H^1(\Bb^\vee)$ and $H^1(\Bb^\vee(-1))$:
\begin{equation}\label{k2} 0\to \Bb^\vee\to \Pp \to \Oo_Q(1)^{\oplus 3} \oplus \Oo_Q^{\oplus 2} \to 0
\end{equation}
$H^2_*(\Pp)\cong H^2_*(\Bb^\vee)$ and  $H^1(\Pp)=H^1(\Pp(-1))=0$. $h^0( \Bb^\vee)-h^0 (\Pp)+ h^0(\Oo_Q(1)^{\oplus 3} \oplus \Oo_Q^{\oplus 2})-h^1(\Bb^\vee)=0$ then $h^0(\Pp)=18$. Moreover $h^3(\Pp(-3))=h^3(\Oo_Q(-3)^{\oplus 2})=2$. So the cohomology table for $\Pp$ and $\Pp^{\vee}$ are as follows:\\
\begin{center}
\begin{tabular}{|c|c|c|c|}
  \hline
  2 & 0 & 0 & 0  \\
  0 & 0 & 0 & 0  \\
  0 & 0 & 0 & 0  \\
  0 & 0 & 0 & 18  \\
   \hline
   \hline
  -3& -2& -1 & 0  \\
  \hline
\end{tabular}\quad\quad
\begin{tabular}{|c|c|c|c|}
  \hline
  18 & 0 & 0 & 0  \\
  0 & 0 & 0 & 0  \\
  0 & 0 & 0 & 0 \\
  0 & 0 & 0 & 2  \\
   \hline
   \hline
  -3& -2& -1 & 0  \\
  \hline
\end{tabular}
\end{center}

Since $\Pp(1)$ and $\Pp^\vee(1) $ are Castelnuovo-Mumford regular, so $\Pp$ is an ACM vector bundle. Note that the rank of $\Pp $ is $10$, $c_1(\Pp)=4$, $h^0(\Pp)=18$ and $h^0(\Pp^\vee)=2$. So we have $\Pp=\Oo_Q^{\oplus 2}\oplus \Sigma^{\oplus 4}$, the sequence (\ref {k2}) reduces to
$$0\to \Bb^\vee\to \Sigma^{\oplus 4} \to \Oo_Q(1)^{\oplus 3} \to 0$$ and the monad is as we claimed.

\end{proof}

\begin{remark}
In \cite{faenzi}, every vector bundle $\Ff$ of rank 2 with the Chern classes $(c_1, c_2)=(-1,k)$ and $H^1(\Ff(-1))=0$ on $Q$, can be shown to be the cohomology of a monad
$$0\rightarrow \Oo_Q^{\oplus k-1} \rightarrow \Sigma^{\oplus k} \rightarrow \Oo_Q(1)^{\oplus k-1} \rightarrow 0$$
using the Kapranov spectral sequences on $Q$.
\end{remark}
Note that $H^1(\Ee(-3))=H^1(\Ii_C)=0$ implies that $H^0(\Oo_C)=1$, i.e. $C$ is a smooth irreducible elliptic curve of degree 8.

\begin{proposition}\label{a00}
There is a globally generated and stable vector bundle $\Ee$ of rank 2 on $Q$ with $(c_1, c_2)=(3,8)$ such that
\begin{itemize}
\item $h^0(\Ee(-1))=0$, $h^0(\Ee)=7$
\item $h^1(\Ee(-1))=2, h^1(\Ee(-2))=3$ and $h^1(\Ee(t))=0$ for
all $t\ge 0$,$t\le -3$
\end{itemize}
with a smooth, non-degenerate and irreducible elliptic curve of degree $8$ as its associated curve on $Q$.
\end{proposition}

Let us fix $5$ distinct points $P_1,\dots P_5\in \mathbb {P}^2$ such that no $3$ of them are collinear. Let
$\pi : W\to \mathbb {P}^2$ be the blowing-up of these 5 points. The anticanonical line bundle $\omega _W^\vee$
of $W$ is very ample and the image $U \subset \mathbb {P}^4$ of $W$ by the complete linear system $\vert \omega _W^\vee \vert$
is a smooth Del Pezzo surface of degree $4$ which is the complete intersection of two quadric hypersurfaces (\cite{d})
(in the standard notation for Del Pezzo surfaces $\omega _W^\vee = (3;1,1,1,1,1)$, i.e. it is given by the strict transform of
all cubic plane curves containing all points $P_1,\dots ,P_5$). Conversely, any smooth complete intersection $U'\subset
\mathbb {P}^4$ is the complete intersection of two quadric hypersurfaces. Hence for general $P_1,\dots ,P_5$
we may assume that $U$ is contained in a smooth quadric. Hence we see $U\subset Q$. Every curve $C\subset U$ is
a curve contained in $Q$. We have $h^0(U,\mathcal {O}_U(1)) = 5$, $h^0(U,\mathcal {O}_U(2)) =13$ and $h^0(U,\mathcal {O}_U(3)) = \chi (O_U) +\mathcal {O}_U(3)\cdot \mathcal {O}_U(4)/2 =25$ by the Riemann-Roch theorem.

\begin{remark}\label{a1}
Let $E\subset \mathbb {P}^4$ be a smooth and non-degenerate elliptic curve such that $\deg (E)
=6$. Since $h^1(\mathcal {I}_{E,\mathbb {P}^4}(1)) =1$, it is easy to check that
$h^1(\mathcal {I}_{E,\mathbb {P}^4}(2)) =0$. Castelnuovo-Mumford's lemma implies that
$\mathcal {I}_{E,\mathbb {P}^4}(3)$ is spanned. Hence for every smooth elliptic curve $E\subset Q$
with $\deg (E)=6$ and $E$ not contained in a hyperplane, the sheaf $\mathcal {I}_E(3)$ on $Q$ is spanned.
\end{remark}

Fix any such smooth elliptic curve $E\subset Q$ such
that $\deg (E)=6$ and $E$ is not contained in a hyperplane of $\mathbb {P}^4$. Since $h^0(E,\mathcal {O}_E(2)) =12 = h^0(\mathcal {O}_{Q}(2))-2$,
it is contained in some quadric hypersurface of $Q$. Let $U\subset Q$ be the smooth Del Pezzo surface of degree $4$ just introduced. We find smooth elliptic curves of degree $6$ inside $U$ by
taking the smooth elements of type $(3;1,1,1,0,0)$. Fix one such curve $E$. Since $h^0(U,\mathcal {O}_U(2))=13$
and $h^0(E,\mathcal {O}_E(2))=12$, $E$ is contained in at least one quadric hypersurface, $T$, of $U$. See $E$ as a curve of type $(3;1,1,1,0,0)$. $T$ has type $(6;2,2,2,2,2)$. Hence the curve $T-E$ is a curve
of type $(3;1,1,1,2,2)$. No cubic plane curve with at least two singular points is integral. We
see that $h^0(U,\mathcal {O}_U(T-U)) =1$ and that the unique curve in $\vert T-U\vert$ is the disjoint union of two lines $R_1$ and $R_2$ with $R_1$ the image of the strict transform in $W$
of the only conic containing the five points $P_1,\dots ,P_5$ (i.e. the only curve of type $(2;1,1,1,1,1)$), while $R_2$ is the image in $U$
of the strict transform of the line of $\mathbb {P}^2$ spanned by $P_4$ and $P_5$ (i.e. the only curve of type $(1;0,0,0,1,1)$).

\begin{remark}\label{a1.1}
Since $\mathcal {I}_E(3)$ is spanned the line bundle $\mathcal {L}:= \mathcal {O}_U(3)(-E)$
is spanned. Let $\alpha : U \to \mathbb {P}^r$, $r:= h^0(U,\mathcal {L})-1$, denote
the morphism induced by $\vert \mathcal {L}\vert$. Since $\mathcal {L}$ has type $(6;2,2,2,3,3)$, we have
$\mathcal {L}\cdot \mathcal {L}
 = 36 -4-4-4-9-9=6$ and $\mathcal {L}\cdot \omega _U^\vee = \mathcal {L}\cdot \mathcal {O}_U(1)
 = 18 -2-2-2-3-3 =6$. Riemann-Roch gives $\chi (\mathcal {L}) =1 +(6+6)/2 =7$. Since $\mathcal {L}$ is spanned, Serre duality
gives $h^2(\mathcal {L})=0$. Hence $r\ge 6$. Since $\alpha (U)$
 spans $\mathbb {P}^r$, we have $\deg (\alpha (U)) \ge 6$. Since
 $\mathcal {L}\cdot \mathcal {L} >0$, $\alpha (U)$ is a surface. Since $\deg (\alpha )\cdot \deg (\alpha (U)) = \mathcal {L}\cdot \mathcal {L} = 6$, we have $\deg (\alpha ) =1$, i.e. $\alpha$
 is birational onto its image.
\end{remark}

\begin{lemma}\label{a2}
If the pair $(O_1,O_2)$ is general in $U\times U$, then $\mathcal {I}_{E\cup \{O_1,O_2\},U}(3)$ is spanned.
\end{lemma}

\begin{proof}
Take $\mathcal {L}$, $\alpha$ and $\mathbb {P}^r$ as in Remark \ref{a1.1}. Lemma \ref{a2} just says that $\{O_1,O_2\}$ is the scheme-theoretic base locus, $\beta$, of
the linear system $\vert \mathcal {I}_{\{O_1,O_2\}}\otimes \mathcal {L}\vert$ on $U$. Since $\alpha$ is birational onto its image and $O_1,O_2$ are general, we have $\alpha (O_1) \ne \alpha (O_2)$
and $\alpha ^{-1}(\alpha (O_i)) =O_i$ as schemes. In characteristic zero a general codimension $2$ section of the non-degenerate surface $\alpha (U)$ is in
linearly general position (\cite{acgh}, pages 112--113). Since the pair $(\alpha (O_1),\alpha (O_2))$
is general in $\alpha (U)\times \alpha (U)$, we get that $\{\alpha (O_1),\alpha (O_2)\}$ is
the scheme-theoretic intersection of $\alpha (U)$ with the line of $\mathbb {P}^r$ spanned by $\alpha (O_1)$ and $\alpha (O_2)$. Hence $\beta = \{O_1,O_2\}$.
\end{proof}

Fix $E\subset U \subset Q$ as above
and a general $(O_1,O_2)\in U\times U$. The union of the set of all lines containing $O_i$
and contained in $Q$
is the quadric cone $(T_{O_i}Q)\cap Q$. Hence there is a line $D_i\subset Q$
containing $O_i$ and intersecting $E$. For general $(O_1,O_2)$ we may assume $D_1\cap D_2 =\emptyset$ and that $D_i$ intersects quasi-transversally $E$ and only at the point $A_i:= D_i\cap E$. Hence $Y:= E\cup D_1\cup D_2$ is a nodal and connected curve of degree $8$ inside $Q$ with $p_a(Y)=1$.

The reduced curve $Y$ is locally a complete intersection and hence
its normal sheaf $\Nn_{Y}$ is a rank 2 vector bundle of degree $\deg (TQ\vert_Y)
= 24$. Since $Y$ has nodes only at $A_1$ and $A_2$, the vector bundles $\Nn_Y\vert_E$
are obtained from $\Nn_E$ making two positive elementary transformations. Since $h^1(E,\Nn_E)=0$,
we have $h^1(E,\Nn_Y\vert_E)=0$. Each normal bundle $\Nn_{D_i}$ is a direct sum
of a degree $0$ line bundle and a degree $1$ line bundle. Hence for every vector bundle $\Ff$ on $D_i$ obtained from $\Nn_{D_i}$ making one negative elementary transformation has no factor of degree $\le -2$. Hence
$h^1(D_i,\Ff)=0$. Hence $h^1(D_1\cup D_2,\Rr)=0$, where $\Rr$ is the vector bundle
obtained from $\Nn_{D_1\cup D_2}$ making the two negative elementary transformations
at $A_1$ and $A_2$ associated to the tangent lines of $E$ at these points.
Hence $h^1(Y,\Nn_Y)=0$ and $Y$ is smoothable inside $Q$ (use \cite{hh}, Theorem 4.1, for $Q$ instead of the smooth 3-fold $\mathbb {P}^3$). We get that the nearby smooth curves, $C$, have $h^1(\Nn_C)=0$ and they
form a $24$-dimensional family smooth at $C$. By semicontinuity the general such $C$
has also $h^1(\mathcal {I}_C(3))=0$. Since in a flat family $\{C_\lambda \}$ of family
with constant $h^1(\mathcal {I}_{C_\lambda }(3))$, the condition ``$\mathcal {I}_{C_\lambda }(3)$
is spanned'' is an open condition, to find a degree $8$ smooth elliptic
curve $C\subset Q$ with $\mathcal {I}_C(3)$ spanned (and hence to complete the case $c_1=3$, $c_2=8$),
it is sufficient to prove that $\mathcal {I}_Y(3)$ is spanned.

\begin{lemma}\label{a3}
The sheaf  $\mathcal {I}_Y(3)$ is spanned.
\end{lemma}

\begin{proof}
Let $\mathcal {B}$ denote the scheme-theoretic base-locus of
the linear system $\vert \mathcal {I}_Y(3)\vert$ on $Q$. We need to prove that $\mathcal {B} = Y$
as schemes. Since $Y\cap U = E\cup \{O_1,O_2\}$ as schemes and $h^1(\mathcal {I}_U(3)) =h^1(\mathcal {O}_{Q}(1))=0$,
$\mathcal {B} \vert_U$ is the scheme-theoretic base locus of the linear system $\vert \mathcal {I}_{\{O_1,O_2\}}\otimes \mathcal {L}\vert $ on $U$.
Lemma \ref{a2} gives $\mathcal {B}\vert_U = E\cup \{O_1,O_2\}$ as schemes. Let
$H\subset \mathbb {P}^4$ be the hyperplane spanned by the lines $D_1$ and $D_2$.
For general $O_1, O_2$ we may assume that $Q_2:= Q\cap H$ is a smooth quadric surface.
Since $U\cup Q_2\in \vert \mathcal {I}_Y(3)\vert$, $h^1(\mathcal {I}_{Q_2}(3))
= h^1(\mathcal {O}_{Q}(2)) =0$ and $\mathcal {B}\vert_U = E\cup \{O_1,O_2\}$, to prove
the lemma it is sufficient to prove that $D_1\cup D_2\cup (E\cap Q_2)$ is the scheme
theoretic base locus of the linear system $\vert \mathcal {I}_{D_1\cup D_2\cup (E\cap Q_2),Q_2}(3)\vert$ on $Q_2$. We call $(1,0)$ the system of lines of $Q_2$ contained $D_1$ and $D_2$.
Since $Y$ is nodal, $D_1\cup D_2\cup (E\cap Q_2) = D_1\cup D_2\cup Z$ with $\deg (Z) =4$
and $E\cap H = \{A_1,A_2\}\cup Z$. Since $\mathcal {I}_{D_1\cup D_2,Q_2}(3)
\cong \mathcal {O}_{Q_2}(1,3)$, it is sufficient that $Z$ is not contained in a line.
Since $\mathcal {I}_E(3)$ is spanned, no line contains a degree $4$ subscheme of $E$.
\end{proof}

\begin{lemma}\label{a4}
Let $Y = E\cup D_1\cup D_2$ as above. For general $Y$ we have
$h^1(\mathcal {I}_Y(3))=0$.
\end{lemma}

\begin{proof}
Let $M\subset \mathbb {P}^4$ be the hyperplane spanned by the lines $D_1$ and $D_2$. Set $Q':= Q\cap M$. Look at the
Castelnuovo exact sequence
\begin{equation}\label{eqc1}
0 \to \mathcal {I}_E(2) \to \mathcal {I}_Y(3) \to \mathcal {I}_{D_1\cup D_2\cap (E\cap Q'),Q'}(3) \to 0
\end{equation}
Since $h^1(\mathcal {I}_E(2))=0$ (Lemma \ref{a1}), the exact sequence (\ref{eqc1}) shows that it is sufficient to prove
that $h^1(Q',\mathcal {I}_{D_1\cup D_2\cap (E\cap Q'),Q'}(3))=0$. Since the integral quadric surface $Q'$ contains
$D_1$ and $D_2$ and $D_1\cap D_2=\emptyset$, $Q'$ is a smooth quadric surface. Call
$(1,0)$ the ruling of $Q'$ containing $D_1$ and $D_2$. For general $D_1$ and $D_2$ the hyperplane $M$ is
not tangent to $E$ either at $A_1$ or at $A_2$. Hence the scheme $E\cap Q'$
is the disjoint union of
$\{A_1,A_2\}$ (with its reduced structure) and a degree $4$ scheme $Z$. Since $D_i\cap E = \{A_i\}$, we
have $Z\cap (D_1\cup D_2)=\emptyset$. Hence it is sufficient to prove $h^1(Q',\mathcal {I}_{Z,Q'}(1,3)) =0$. Since
$\deg (Z) = 4$, it is easy to check that $h^1(Q',\mathcal {I}_{Z,Q'}(1,3))>0$ if and only if
there is a scheme $B\subset Z$ with $\deg (B)=3$ and $B$ in a line of type $(1,0)$ on $Q'$. To exclude the existence
of such a scheme $B$ it is sufficient to find $E$ without a two-dimensional family of trisecant lines (move $D_1$ and $D_2$).
\end{proof}

\vspace{0.3cm}

\quad {\emph {Proof of Proposition \ref{a00}.}} By the Serre correspondence it is sufficient to find a smooth
elliptic curve $C\subset Q$ such that $\mathcal {I}_C(3)$ is spanned, $h^1(\mathcal {I}_C(t))=0$ for all $t\ge 3$
and $h^0(\mathcal {I}_C(2)) =0$ (indeed, the last condition
implies $h^0(\mathcal {I}_C(t))=0$ for all $t \le 1$ and hence $h^1(\mathcal {I}_C(2)) =2$, $h^1(\mathcal {I}_C(1))=3$ and $h^1(\mathcal {I}_C(t))=0$ for all $t\le 0$ by the Riemann-Roch theorem). By
semicontinuity it is sufficient to find $Y$ with the same properties. There is $Y$ with $\mathcal {I}_Y(3)$ spanned
(Lemma \ref{a3}) and with $h^1(\mathcal {I}_Y(3))=0$ (Lemma \ref{a3}). Castelnuovo-Mumford's lemma implies
$h^1(\mathcal {I}_Y(t))=0$ for all $t\ge 4$. Since $Y =E\cup D_1\cup D_2$ with $h^0(\mathcal {I}_E(2))=0$
(Lemma \ref{a1}), we have $h^0(\mathcal {I}_Y(2))=0$.\qed

\begin{remark}
For a vector bundle $\Ee$ in $\mathfrak{M}(3,8)$ that fits into the sequence:
$$0\rightarrow \Oo_Q(1) \rightarrow \Ee \rightarrow \Ii_Z(2)\rightarrow 0,$$
where $Z$ is the disjoint union of 4 lines, we can easily compute that $h^2(\mathcal{E}nd(\Ee))=0$. It implies that $\Ee$ is a smooth point of $\mathfrak{M}(3,8)$ and the dimension of $\mathfrak{M}(3,8)$ is 18 since such bundles form a 15-dimensional subvariety of $\mathfrak{M}(3,8)$ and $\chi(\mathcal{E}nd (\Ee))=-17$.
\end{remark}
\section{Case $(c_1, c_2)=(3,9)$}
Assume that $c_2(\Ee)=9$. Again let $\Ff=\Ee(-2)$. We can compute:\\
$$
\left\{
       \begin{array}{ll}
       H^1(\Ff(t))=0 \text{  for  }t\leq -2\\
       \chi(\Ff\otimes \Ff^{\vee})=-23\\
       \chi(\Ff)=-\chi (\Ff(-2))=-4\\
       \chi(\Ff(1))=-\chi(\Ff(-3))=-4\\
       \chi(\Ff(-1))=0\\
       \chi(\Ff(2))=4
                                           \end{array}
                             \right. $$
Then the cohomology table for $\Ff=\Ee(-2)$ is as follows:
\begin{equation}\label{t2}\begin{tabular}{|c|c|c|c|c|c|}
  \hline
  0 & 0 & 0 & 0 & 0 & 0 \\
  4 & 4 & c & 0 & 0 & 0 \\
  0 & 0 & c & 4 & 4 & $b$ \\
  0 & 0 & 0 & 0 & 0 & $a$ \\
   \hline
   \hline
  -3& -2& -1 & 0 & 1 & 2 \\
  \hline
\end{tabular}\end{equation}
with $a-b=4$. If $\{k_1\leq \cdots \leq k_4\}$ is the spectrum of $\Ff$, then the possibility is either $\{-2,-1,-1,0\}$ or $\{ -1,-1,-1,-1\}$. It implies that $H^1(\Ee(-3))=0$ or $1$, i.e. $c=0$ or $1$ in the table. Thus $h^0(\Oo_C)$ is either 1 or 2. Since $\omega_C\simeq \Oo_C$, so $C$ is either an irreducible elliptic curve of degree 9 or consists of two irreducible elliptic curves. Note that $C$ cannot have a plane cubic curve as its component.
\begin{lemma}
$C$ is an irreducible smooth elliptic curve of degree 9.
\end{lemma}
\begin{proof}
Assume that $C=C_1\sqcup C_2$ with $C_i$ smooth elliptic, $\deg (C_1) =4$ and $\deg (C_2) = 5$.
The curve $C_1$ is a contained in a hyperplane section $J$, which may be a cone.
Even if $J$ is not smooth, $C_1$ is a complete intersection of $J$ with a quadric surface
in the linear span $\langle J\rangle \cong \mathbb {P}^3$. So in $J\cap (C_1\cup C_2)$
we have $C_1$ and a degree $5$ scheme $J\cap C_2$; the only possibility
to have $\mathcal {I}_{C_1\cup C_2}(3)$ spanned is that $\mathcal {I}_{J\cap C_2,J}(1)$
is spanned, which is absurd.
\end{proof}
In particular, we have $c=0$ in the cohomology table of $\Ff$.
Similarly as in the proposition \ref{monad}, we can show that $\Ee$ is the cohomology of a monad:
$$0\rightarrow \Oo_Q(1)^{\oplus 4} \rightarrow \Sigma(1)^{\oplus 5} \rightarrow \Oo_Q(2)^{\oplus 4} \rightarrow 0.$$

Now we will prove the existence of a smooth elliptic curve $C\subset Q$ such that $\deg (C) = 9$
and $h^1(\mathcal {I}_C(3))=0$. Since $3\cdot 9 = 27 = h^0(\mathcal {O}_{Q}(3))-3$, we have
$h^1(\mathcal {I}_C(3))=0$ if and only if $h^0(\mathcal {I}_C(3))=3$. The latter condition obviously implies $h^0(\mathcal {I}_C(2)) =0$. Hence proving the existence of $C$ also proves the existence
of a rank two vector bundle $\Ee$ on $Q$ with $(c_1, c_2)=(3,9)$,
$h^0(\Ee)=4$, $h^1(\Ee)=0$, $h^0(\Ee(-1))=0$ (and hence $\Ee$ is stable).

\begin{lemma}\label{b1}
There is a smooth elliptic curve $A\subset Q$ such that $\deg (A) =7$ and $h^0(\mathcal {I}_A(2))=0$.
\end{lemma}

\begin{proof}
We start with a smooth hyperplane section $Q_2$ of $Q$ and a smooth elliptic curve $B \subset
Q_2$ of type $(2,2)$. Fix $2$ general points $P_1,P_2\in B$. The union $U(P_i)$
of all lines in $Q$ passing through $P_i$ is the quadric cone $T_{P_i}(Q)\cap Q$
of the tangent space $T_{P_i}(Q) \cong \mathbb {P}^3$ of $Q$ at $P_i$. Hence we may find a
line $D_1\subset Q$ such that $P_1\in D_1$ and a smooth conic
$D_2 \subset Q$ such that $D_2\cap Q_2$ is $P_2$ and a point not in $B$. For general $D_i$
we may also assume $D_1\cap D_2=\emptyset$ and $D_i\nsubseteq Q_2$.  For general $D_1, D_2$ we may also
assume that no hyperplane section of $Q$ contains $D_1\cup D_2$. Set $Y:= B\cup D_1\cup D_2$. Since $Q_2$ is a hyperplane section of $Q$, $D_i\nsubseteq Q_2$ and $B\subset Q_2$, then
$Q_2\cap D_i = \{P_i\}$ and $D_i$ is not the tangent line of $B$ at $P_i$. Hence $Y$ is a nodal connected curve of degree $7$ and with arithmetic genus $1$.

The reduced curve $Y$ is locally a complete intersection and hence
its normal sheaf $\Nn_{Y}$ is a rank 2 vector bundle of degree $\deg (TQ\vert_Y)
= 21$. Since $Y = B\cup (D_1\cup D_2)$ has nodes only at $P_1,P_2$, the vector bundles $\Nn_Y\vert_B$
are obtained from $\Nn_B$ making $2$ positive elementary transformations, each of them at one of the points $P_i$. Since $\Nn_B$
is a direct sum of a line bundle of degree $4$ and a line bundle of degree $8$, we have $h^1(B,\Nn_B)=0$.
Hence $h^1(B,\Nn_Y\vert_B)=0$. The normal bundle $\Nn_{D_1}$ is a direct sum
of a degree $0$ line bundle and a degree $1$ line bundle. The normal bundle $\Nn_{D_2}$ is a direct sum
of a degree $2$ line bundle and a degree $4$ line bundle. Hence for every vector bundle $\Ff$ on $D_i$ obtained from $\Nn_{D_1\cup D_2}$ making $2$ negative elementary transformations,
each of them at a different point $P_1,P_2$, has no factor of degree $\le -2$. Hence
$h^1(D_i,\Ff)=0$. Hence $h^1(D_1\cup D_2,\Gg)=0$, where $\Gg$ is the vector bundle
obtained from $\Nn_{D_1\cup D_2}$ making the $2$ negative elementary transformations
at $P_1$ and $P_2$ associated to the tangent lines of $B$ at these points.
Hence $h^1(Y,\Nn_Y)=0$ and $Y$ is smoothable inside $Q$ (use \cite{hh}, Theorem 4.1, for $Q$ instead of the smooth 3-fold $\mathbb {P}^3$). By semicontinuity to prove Lemma \ref{b1} it
is sufficient to prove $h^0(\mathcal {I}_Y(2))=0$. Assume $h^0(\mathcal {I}_Y(2))>0$ and
take $\Delta \in \vert \mathcal {I}_Y(2)\vert$. Since $Q_2\cap Y$ contains $B$ and a point of $D_2\cap (Q_2\setminus B)$, we have $Q_2\subset \Delta$, i.e. $\Delta = Q_2\cup Q'$
for some hyperplane section $Q'$ of $Q$. Since neither $D_1$ nor $D_2$ is contained
in $Q_2$, we get $D_1\cup D_2\subset Q'$, contradicting our choice of $D_1\cup D_2$.
\end{proof}

\begin{lemma}\label{b2}
There is a smooth elliptic curve $C\subset Q$ such that $\deg (C) = 9$, $h^0(\mathcal {I}_C(3))=3$
and $h^1(\mathcal {I}_C(3))=0$.
\end{lemma}

\begin{proof}
Let $C\subset Q$ be any smooth elliptic curve of degree $9$. Since $h^0(\mathcal {O}_C(3))
= 27 = h^0(\mathcal {O}_{Q}(3))-3$, we have $h^1(\mathcal {I}_C(3)) = h^0(\mathcal {I}_C(3))-3$.
Hence it is sufficient to prove the existence of a smooth elliptic curve $C$
with $ \deg (C) = 9$ and $h^0(\mathcal {I}_C(3))=3$. Let $A\subset Q$ be a smooth elliptic curve such that $\deg (A) =7$ and $h^0(\mathcal {I}_A(2))=0$. Fix two general $P_1, P_2\in A$
(Lemma \ref{b1}). The union $U(P_i)$
of all lines in $Q$ passing through $P_i$ is the quadric cone $T_{P_i}(Q)\cap Q$
of the tangent space $T_{P_i}(Q)$ of $Q$ at $P_i$. Hence we may find lines $D_i\subset Q$, $i=1,2$, such that $P_i\in D_i$, $D_i$ is not the tangent line to $A$ at $P_i$, $D_1\cap D_2=\emptyset$ and $D_i\cap A = \{P_i\}$. Hence $Y:= A\cup D_1\cup D_2$ is a connected nodal curve
with degree $9$ and arithmetic genus $1$. As in the proof of Lemma \ref{b1} we see that $Y$ is smoothable inside $Q$. Hence by semicontinuity to prove Lemma \ref{b2} it is sufficient
to prove $h^0(\mathcal {I}_Y(3))=3$. Let $H\subset \mathbb {P}^4$ be the hyperplane spanned
by $D_1\cup D_2$. Set $Q':= Q\cap H$. Since $Q'$ contains the disjoint lines $D_1, D_2$, $Q'$
is a smooth quadric surface. Call $(1,0)$ the ruling of $Q'$ containing $D_1$ and $D_2$.
Fix general $O_1,O_2,O_3\in Q'$. Since $h^0(\mathcal {I}_Y(3)) \ge 3$ by the Riemann-Roch theorem,
to prove $h^0(\mathcal {I}_Y(3))=3$ it is sufficient to prove $h^0(\mathcal {I}_{Y\cup \{O_1,O_2,O_3\}}(3)) =0$. Assume $h^0(\mathcal {I}_{Y\cup \{O_1,O_2,O_3\}}(3)) >0$
and take $\Delta \in \vert \mathcal {I}_{Y\cup \{O_1,O_2,O_3\}}(3)\vert$. For general $D_1$ and $D_2$ we may also assume that $H$ is not tangent to $A$ neither at $P_1$ nor at $P_2$.
Hence the scheme $A\cap Q'$ is the disjoint union of $P_1$, $P_2$, and a degree $5$ scheme, $Z$, such
that $Z\cap (D_1\cup D_2)=\emptyset$. First assume
$Q'\subset \Delta$. Hence $\Delta = Q'\cup T$ with $T$ a quadric hypersurface of $Q$.
Since $Q'\cap A$ is a finite set, we get $A\subset T$. Hence $h^0(\mathcal {I}_A(2)) >0$, a contradiction.
Hence $\Delta' := \Delta \cap Q'$ is a divisor of type $(3,3)$ of $Q'$ containing $D_1\cup D_2$.
Set $J:= \Delta ' -D_1-D_2$. $J$ is a divisor of type $(1,3)$ on $Q'$ containing $Z$, $O_1$, $O_2$ and $O_3$. Since the points $O_i$ are general in $Q'$, to get a contradiction (and hence to prove
the lemma) it is sufficient to prove $h^0(Q',\mathcal {I}_Z(1,3)) =3$, i.e. $h^1(\mathcal {I}_Z(1,3))=0$. Fix a smooth $D\in \vert \mathcal {O}_{Q'}(1,1)\vert$ and take a general hyperplane section
$T$ of $Q$ with $T\cap T' = D$. $T$ is a smooth quadric surface. We may specialize $A$ to a curve $Y':= A'\cup L_1\cup L_2\cup L_3$ with $A'$ a smooth curve of type $(2,2)$, each $L_i$ a line
intersecting transversally $A'$ and $P_i\in L_i$, $i=1,2$. For general $Y'$ we have
$Y\cap Q' = Z'\cup \{P_1,P_2\}$ with $\sharp (Z'\cap D) =4$. By semicontinuity it is sufficient to prove
$h^1(Q',\mathcal {I}_{Z'}(1,3)) =0$. Since $\sharp (Z'\setminus Z'\cap D)=1$, we
have $h^1(Q',\mathcal {I}_{Z'\setminus Z'\cap D;Q'}(0,2))=0$. $D$ is a smooth rational curve and $\deg (\mathcal {O}_D(1,3))=4$. Hence $h^1(D,\mathcal {I}_{D\cap Z',S}(1,3))=0$. From the exact sequence on $Q'$: $$0\to \mathcal {I}_{Z'\setminus Z'\cap D,Q'}(0,2) \to \mathcal {I}_{Z',Q'}(1,3) \to \mathcal {I}_{D\cap Z',D}(1,3)\to 0$$
we get $h^1(Q',\mathcal {I}_{Z'}(1,3)) =0$, concluding the proof.
\end{proof}

Let us start with the smooth elliptic curve $C$ given by Lemma \ref{b2} ; hence $h^1(\mathcal {I}_C(3))=0$
and $h^0(\mathcal {I}_C(3))=3$. Since $h^0(\Ii _C(3)) \ge 2$, there are at least two degree $3$ hypersurfaces $M_1, M_2$ containing
$C$. Since $h^0(\Ii _C(2))=0$, the scheme-theoretic intersection $X:= M_1\cap M_2$ has
dimension $1$, i.e. it is a complete intersection.
 By the Liaison induced by $X$, we have
 \begin{equation}\label{eqa2}
0\to \Ii_X(3) \to \Ii_C(3) \to \omega _Y \to 0.
\end{equation}
We need to prove that $\omega _Y$ is spanned. Since $X$ is a complete intersection of $M_1$
and $M_2$, we have
$h^0(\Ii _X(3))=2$ and $h^1(\Ii _X(3))=0$. Since $h^0(\Ii _C(3)) =3$, the sequence (\ref{eqa2}) gives $h^0(\omega _Y)=1$. We have $h^2(\Ii _C(3)) = h^1(\Oo _C(3))=0$ and $h^2(\Ii _X(3))
= h^1(\Oo_X(3)) =1$ since $\omega _X \cong \Oo _X(3)$ by the adjunction formula.
Since $h^1(\Ii _C(3)) =0$ we get $h^1(\omega _Y)=1$.
Hence $h^0(\mathcal {O}_Y)=1$ by the duality of locally Cohen-Macaulay projective
schemes. Since $p_a(Y)=1$ to get that
$\omega _Y$ is trivial and hence that $\omega _Y$ is spanned, it is sufficient to prove that (at least for
certain $C$) it is spanned outside finitely many points.
Since $\omega _Y$ is a quotient of $\mathcal {I}_C(3)$, it is spanned at all
points at which $\mathcal {I}_C(3)$ is spanned. Hence it is sufficient to find $C$ with the
additional condition that $\mathcal {I}_C(3)$ is spanned outside finitely many points. This is the Lemma
\ref{b3} below.

\begin{lemma}\label{b3}
There is a smooth elliptic curve $C\subset Q$ such that $\deg (C) = 9$, $h^0(\mathcal {I}_C(3))=3$, $h^1(\mathcal {I}_C(3))=0$ and such that $\mathcal {I}_C(3)$ is spanned outside
finitely many points.
\end{lemma}

\begin{proof}
By semicontinuity it is sufficient to find $Y = B\cup D_1\cup D_2$ as in the proof of Lemma \ref{b2} with the additional property that the base locus of $\mathcal {I}_Y(3)$ is finite.
Fix $B$ satisfying the thesis of Lemma \ref{b1}. Let $H\subset \mathbb {P}^4$ be a general hyperplane.
By the Uniform Position Principle (\cite{acgh},  pp. 112--113) the scheme $B\cap H$ is formed by $7$
points in uniform position and we call $A_1, A_2$ two of them. Moreover, the
monodromy of the generic hyperplane section is the full transite group (\cite{acgh}, p. 112).
Hence for general $H$ we may assume that no two of the points of $A\cap H$ are contained
in a line in $H$. Set $Q':= Q\cap H$. For
general $H$ the scheme $Q'$ is a smooth quadric surface. Fix one of the system of lines
of $Q'$, say $(1,0)$, and call $D_i$ the line of type $(1,0)$ of $Q'$ containing $A_i$. Set $S:= B\cap Q'\setminus \{A_1,A_2\}$. Notice that $\sharp (S)=5$ and $S\cap D_i=\emptyset$.
Set $Y:= B\cup D_1\cup D_2$. The proof of Lemma \ref{b2} gives $h^1(\mathcal {I}_Y(3))=0$
and $h^0(\mathcal {I}_Y(3))=3$. Let $\mathcal {B}$ denote the base locus
of $\mathcal {I}_Y(3)$. It is sufficient to prove that $\mathcal {B}\cap Q'=\emptyset$. Since
$h^i(\mathcal {I}_B(2))=0$, $i=0,1$, we saw in the proof of Lemma \ref{b2}
that the restriction of $\vert \mathcal {I}_Y(3)\vert$ to $Q'$ is given by all forms
$D_1\cup D_2\cup T$ with $T\in \vert \mathcal {I}_S(1,3)\vert$
and $h^0(Q',\mathcal {I}_S(1,3))=3$. Since $S$ is in uniform position, $T$ is a general element
of $\vert \mathcal {I}_S(1,3)\vert$. Fix two general $T_1, T_2\in \vert \mathcal {I}_S(1,3)\vert$. Since
$T_1$ is irreducible, the scheme $T_1\cap T_2$ is zero-dimensional. We have $\deg (T_1\cap T_2) = 3+3 =6$. Since $S\subset T_1\cap T_2$ and $h^0(Q',\mathcal {I}_S(1,3))=3$, we get
that $\mathcal {I}_S(1,3)$ is a spanned sheaf of $Q'$. Since $S\cap (D_1\cup D_2)=\emptyset$,
we also get that the scheme $Y\cap Q'$ is the intersection with $Q'$ of all elements
of $\vert \mathcal {I}_Y(3)\vert$. Hence $\mathcal {B}\cap Q'=\emptyset$. Hence $\mathcal {B}$
is supported by finitely many points.
\end{proof}

\begin{remark}
Similarly as in the case $c_1=8$, for a vector bundle $\Ee \in \mathfrak{M} (3, 9)$ that fits into the sequence:
$$0\rightarrow \Oo_Q(1) \rightarrow \Ee \rightarrow \Ii_Z(2)\rightarrow 0,$$
where $Z$ is the disjoint union of 5 lines, we can easily compute that $h^2(\mathcal{E}nd(\Ee))=0$. It implies that $\Ee$ is a smooth point of $\mathfrak{M}(3,9)$ and the dimension of $\mathfrak{M}(3,9)$ is 24 since such bundles form a 19-dimensional subvariety of $\mathfrak{M}(3,9)$ and $\chi(\mathcal{E}nd(\Ee))=-23$.
\end{remark}

As an automatic consequence from the classification, we observe that if $\Ee$ is a rank two globally generated vector bundle on $Q$ with $c_1=3$, then we have $c_1(\Ee(-2))=-1$ and $H^1(\Ee(-3))=0$. Thus we have the following:

\begin{corollary}
Every rank two globally generated vector bundle on $Q$ with $c_1=3$, is an odd instanton (see \cite{faenzi}) up to twist.
\end{corollary}


\section{On higher dimensional quadrics}

We denote by $Q_n$ the smooth quadric hypersurface of dimension $n>3$. 

\begin{theorem}\label{aBonQn}
The only indecomposable and globally generated vector bundles $\Ff$ of rank $2$ on $Q_n$, $n \ge 4$, with $c_1(\Ff)\le3$  are the followings:
\begin{enumerate}
\item for $n\ge6$, no such bundle exists,
\item for $n=5$, $\Ff(-2)$ is a Cayley bundle, i.e.\ $\Ff(-2)$ is a bundle with $c_1=-1$, $c_2=2$,
\item for $n=4$,  either $\Ff(-2)$ is a spinor bundle 
or it has $c_1=-1$, $c_2=(1,1)$, i.e.\ $\Ff(-2)$ is the restriction of a Cayley bundle to $Q_4$.
\end{enumerate}
\end{theorem}
\begin{proof}
We have to study which vector bundle of Theorem \ref{mt} extends to a globally generated vector bundle on $Q_n$:

\quad{(a)} If $\Ee$ is a spinor bundle on $Q_3$, then $\Ee$ extends to a spinor bundle $\Sigma_1$ or $\Sigma_2$ on $Q_4$, but does not extend on $Q_n$ for $n\ge5$ (see \cite{O1}, theorem 2.1).

\quad{(b)} If $\Ee$ is a pullback of a null-correlation bundle on $\mathbb {P}^3$ twisted by $1$, the zero locus of a general section is the disjoint union of two conics.
Let assume that $\Ff$ is a globally generated extension of $\Ee$ on $Q_4$. Then the zero locus of a general section must be the disjoint union of two quadric surfaces. We recall that a quadric surface $S\subset Q_4$ is the complete intersection of two
hyperplane sections of $Q_4$. Hence any two quadric surfaces meet so we get a contradiction.  

\quad{(c)} If $\Ee$ is stable with $c_1=-1$, $c_2=2$, then $\Ee$ extends to a vector bundle on $Q_4$ with Chern classes $c_1=-1$, $c_2=(1,1)$, and even to one on $Q_5$ with $c_1=-1$, $c_2=2$ (Cayley bundles), but no further on $Q_n$, $n\ge6$, (see \cite{Ott}, theorem 3.2).

\quad{(d)} Let $\Ee$ be a rank $2$ globally generated on $Q$  with  $c_1(\Ee)=3$ and $c_2(\Ee)=7$. Then he does not extend on $Q_4$ by  \cite{bmvv} Theorem $4.3$.

\quad{(e)} Let $\Ff$ be a rank $2$ globally generated on $Q_4$ such that the restriction $\Ee=\Ff|_H$ of a general hyperplane $H$  is a globally generated vector bundle of rank $2$ on $Q_3$, with  $c_1(\Ee)=3$ and $c_2(\Ee)=8 $. Therefore the zero locus of a general global section of $\Ff$ is a smooth surface $S$ of degree $8$ and we have the exact sequence
$$0 \to \Oo_{Q_4} \to \Ff \to \Ii_S(3) \to 0.$$
Since $\det(\Ff)\cong\Oo_{Q_4}(3)$ and $\omega_{Q_4}\cong\Oo_{Q_4}(-4)$, the adjunction formula gives $\omega_S\cong \Oo_S(-1)$. Thus $S$ is an anticanonically embedded del Pezzo surface. Using Riemann-Roch on the surface $S$ we obtain $h^0(S,\Oo_S(1))=\chi(\Oo_S(1)) = (\Oo_S(1)\cdot\Oo_S(2))/2 + \chi(\Oo_S) = 8 + 1 = 9$. On the other hand, using the structure sequence $0 \to \Ii_S(1) \to \Oo_{Q_4}(1) \to \Oo_S(1) \to 0$, we have $h^1(\Ii_S(1))=h^1(\Ff((-2)) = 3$. Since $h^1(\Oo_S(1))=0$ we get also $h^2(\Ff(-2))=0.$ Now let us consider the sequence,
$$ 0 \to \Ff(-1) \to \Ff \to \Ee \to 0.$$
From the cohomology table for $\Ee(-2)$ (see (\ref{t1})) we get $h^1(\Ff(t))=0$ for any $t\leq -3$ and from the sequence in cohomology
$$0= H^1(\Ff(-3))\to H^1(\Ff(-2))\to H^1(\Ee(-2))=\CC^{\oplus 3}\to  H^2(\Ff(-3))\to 0,$$ we get also $H^1(\Ff(-3))=0$.
For any $t\geq -3$ we have the sequence $$H^2(\Ff(t))\to H^1(\Ff(t+1))\to 0$$ which implies $H^2(\Ff(t+1))=0$. For any $t\leq -3$ we have the sequence $$0\to H^2(\Ff(t-1))\to H^1(\Ff(t))$$ which implies $H^2(\Ff(t-1))=0$. Then we can conclude that $H^2_*(\Ff)=0$ which a contradiction to the classification of rank two vector bundles without $H^2_*$ given in \cite{ma}.

\quad{(f)} Let $\Ff$ be a rank $2$ globally generated on $Q_4$ such that the restriction $\Ee=\Ff|_H$ of a general hyperplane $H$  is a globally generated vector bundle of rank $2$ on $Q_3$, with  $c_1(\Ee)=3$ and $c_2(\Ee)=9 $.
 Therefore the zero locus of a general global section of $\Ff$ is a smooth surface $S$ of degree $9$ and we have the exact sequence
$$0 \to \Oo_{Q_4} \to \Ff \to \Ii_S(3) \to 0.$$
Since $\det(\Ff)\cong\Oo_{Q_4}(3)$ and $\omega_{Q_4}\cong \Oo_{Q_4}(-4)$, the adjuction formula gives $\omega_S\cong \Oo_S(-1)$. Thus $S$ is an anticanonically embedded del Pezzo surface. Using Riemann-Roch on the surface $S$ we obtain $h^0(S,\Oo_S(1))=\chi(\Oo_S(1)) = (\Oo_S(1)\cdot\Oo_S(2))/2 + \chi(\Oo_S) = 9 + 1 = 10$. On the other hand, using the structure sequence $0 \to \Ii_S(1) \to \Oo_{Q_4}(1) \to \Oo_S(1) \to 0$, we have $h^1(\Ii_S(1))=h^1(\Ff((-2)) = 4$. Since $h^1(\Oo_S(1))=0$ we get also $h^2(\Ff(-2))=0.$ Now let us consider the sequence,
$$ 0 \to \Ff(-1) \to \Ff \to \Ee \to 0.$$ From
the cohomology table for $\Ee(-2)$ (see (\ref{t2})) we get $h^1(\Ff(t))=0$ for any $t\leq -3$ and from the sequence in cohomology
$$0= H^1(\Ff(-3))\to H^1(\Ff(-2))\to H^1(\Ee(-2))=\CC^{\oplus 4} \to  H^2(\Ff(-3))\to 0,$$ we get also $H^1(\Ff(-3))=0$.
For any $t\geq -3$ we have the sequence $$H^2(\Ff(t))\to H^1(\Ff(t+1))\to 0$$ which implies $H^2(\Ff(t+1))=0$. For any $t\leq -3$ we have the sequence $$0\to H^2(\Ff(t-1))\to H^1(\Ff(t))$$ which implies $H^2(\Ff(t-1))=0$. Then we can conclude that $H^2_*(\Ff)=0$ which a contradiction with the classification of rank two vector bundles without $H^2_*$ given in \cite{ma}.
\end{proof}

\providecommand{\bysame}{\leavevmode\hbox to3em{\hrulefill}\thinspace}
\providecommand{\MR}{\relax\ifhmode\unskip\space\fi MR }
\providecommand{\MRhref}[2]{%
  \href{http://www.ams.org/mathscinet-getitem?mr=#1}{#2}
}
\providecommand{\href}[2]{#2}

\end{document}